\documentclass[11pt]{article}
\usepackage{enumerate}
\usepackage{amssymb,a4wide,latexsym,makeidx,epsfig,fleqn}
\usepackage{amsthm}
\usepackage{amsmath}
\usepackage{enumerate}
\usepackage{color}
\newtheorem{theorem}{Theorem}[section]
\newtheorem{lemma}[theorem]{Lemma}

\newtheorem{remark}[theorem]{Remark}
\newtheorem{example}[theorem]{Example}

\begin{document}
\textwidth 150mm \textheight 225mm
\title{Iota energy orderings of bicyclic signed digraphs \thanks{Supported by the National Natural Science Foundation of China (No. 11871398), the Natural Science Basic Research Plan in Shaanxi Province of China (Program No. 2018JM1032), the Fundamental Research Funds for the Central Universities (No. 3102019ghjd003), and the Seed Foundation of Innovation and Creation for Graduate Students in Northwestern Polytechnical University (No. ZZ2019031).}}
\author{{Xiuwen Yang$^{a,b}$, Ligong Wang$^{a,b}$\footnote{Corresponding author.}}\\
{\small $^{a}$ School of Mathematics and Statistics, Northwestern Polytechnical University,}\\{\small  Xi'an, Shaanxi 710129, P.R. China.}\\ {\small $^{b}$ Xi'an-Budapest Joint Research Center for Combinatorics, Northwestern Polytechnical University,}\\{\small  Xi'an, Shaanxi 710129, P.R. China.}\\{\small E-mail: yangxiuwen1995@163.com, lgwangmath@163.com}}
\date{}
\maketitle
\begin{center}
\begin{minipage}{135mm}
\vskip 0.3cm
\begin{center}
{\small {\bf Abstract}}
\end{center}
{\small The concept of energy of a signed digraph is extended to iota energy of a signed digraph. The energy of a signed digraph $S$ is defined by $E(S)=\sum_{k=1}^n|\textnormal{Re}(z_k)|$, where $\textnormal{Re}(z_k)$ is the real part of eigenvalue $z_k$ and $z_k$ is the eigenvalue of the adjacency matrix of $S$ with $n$ vertices, $k=1,2,\ldots,n$. Then the iota energy of $S$ is defined by $E(S)=\sum_{k=1}^n|\textnormal{Im}(z_k)|$, where $\textnormal{Im}(z_k)$ is the imaginary part of eigenvalue $z_k$. In this paper, we consider a special graph class for bicyclic signed digraphs $\mathcal{S}_n$ with $n$ vertices which have two vertex-disjoint signed directed even cycles. We give two iota energy orderings of bicyclic signed digraphs, one is including two positive or two negative directed even cycles, the other is including one positive and one negative directed even cycles. \vskip 0.1in \noindent {\bf Key Words}: \ Orderings; Iota energy; Bicyclic signed digraphs \vskip
0.1in \noindent {\bf AMS Subject Classification (2010)}: \ 05C35, 05C50 }
\end{minipage}
\end{center}

\section{Introduction }

Graph energy was defined as the sum of the absolute values of all eigenvalues of the adjacency matrix of a graph. In 1978, Gutman \cite{Gut} first introduced the concept of energy of a graph. Since then, more and more scholars pay attention to graph energy. At present, the studies about other energies of graphs mainly include (signless) Laplacian energy \cite{PiTr,DaFP}, matching energy \cite{GuWa,LiWX}, distance energy \cite{DiRo,Il}, etc. More research results can be found in the papers \cite{DaGu,Gut01} and two good books  \cite{GuLi,LiSG}.

In 2010, Adiga et al. \cite{AdBS} extended the energy of a graph to skew energy of an oriented graph. Thus the skew energy was defined as the sum of the absolute values of all eigenvalues of the skew adjacency matrix of an oriented graph. In 2017, Guo et al. \cite{GuWX} determined all connected $5$-regular oriented graphs with maximum skew-energy. In 2018, Deng et al. \cite{DeLS} found all tournaments achieving the minimum skew energy. In 2019, Wang and Fang \cite{WaFa} established an upper bound for skew energy of an oriented graph. More results on the skew energy of an oriented graph were refered to the survey "Skew energy of oriented graphs" by Li and Lian on page 191--236 of \cite{GuLi}.

In 2008, Pe{\~{n}}a and Rada \cite{PeRa} extended graph energy to digraph energy which was defined as the sum of the absolute values of the real parts of its eigenvalues of the adjacency matrix of a digraph. Moreover, Bhat \cite{Bh} studied the spectra of weighted digraphs and obtained some upper bound for the energy of weighted digraphs. Xi and Wang \cite{XiWa} considered the signless Laplacian energy of digraphs and derived some lower and upper bounds for the signless Laplacian energy of digraphs. The references \cite{KhFa1,MoRa,YaWa} were also described some results concerning digraph energy. For more results, please saw the survey "Energy of digraphs" by Rada on page 237--276 of \cite{GuLi}.

In 2014, the notion of energy of a sigraph was extended to sidigraph by Pirzada and Bhat \cite{PiBh}. The authors also studied the energy of sigraphs in \cite{BhPi} and they characterized the unicyclic sigraphs with minimal energy. Another results about the energy of sigraphs can be found in \cite{BhSP,WaHo}. They found the maximal energy among all $n$-vertices connected bicyclic sigraphs with at most one odd cycle is $\mathcal{P}_n^{4,4}$. For the energy of sidigraph, the minimal and maximal energy of bicyclic sidigraphs were found by Khan and Farooq \cite{KhFa2}. The maximal energy of all bicyclic sidigraphs was found and few classes of non-cospectral equienergetic bicyclic sidigraphs were constructed by Hafeez et al. \cite{HaFK}. The ordering of bicyclic sidigraphs in $\mathcal{S}_n$ by energy with two positive or negative directed even cycles (resp., one positive directed even cycle and one negative directed even cycle) was characterized by Yang and Wang \cite{YaWa2}.

In addition, the concept of iota energy of a digraph was introduced by Khan et al. \cite{KhFR}. Thus the iota energy of a digraph was defined as the sum of the absolute values of the imaginary parts of its eigenvalues of the adjacency matrix of a digraph. In 2019, Farooq et al. \cite{FaKC} extended the notion of iota energy to a sidigraph. They found minimal and maximal values of the iota energy over the set of unicyclic sidigraphs and discussed the increasing property of iota energy over specific subclasses of the set $S_{n,h}$. Farooq et al. \cite{FaCK} found sidigraphs with extremal iota energy among vertex-disjoint bicyclic sidigraphs of fixed order. Another more results of iota energy are studied in \cite{FaKA,HaKh,YaWa}. In this paper, we characterize two orderings of bicyclic sidigraphs in $\mathcal{S}_n$ with iota energy and determine extremal iota energy in $\mathcal{S}_n$ by the two orderings.

The organizational structure of this paper is as follows: In Section 2, we introduce some basic concepts about sidigraphs and the energy of sidigraphs. We also give some results about linear combinations of trigonometric functions can be used in our main results. In Section 3, we have the two iota energy orderings of bicyclic sidigraphs, one is the whole ordering including two positive or two negative directed even cycles, and the other is the ordering including one positive and one negative directed even cycles (except $C_2^+$). We also find the extremal iota energy of bicyclic sidigraphs in $\mathcal{S}_n$ by their two orderings.

\section{Preliminaries}

A digraph $D=(\mathcal{V}, \mathcal{A})$ is defined to be a vertex set $\mathcal{V}$ and an arc set $\mathcal{A}$. A sidigraph $S=(D, \sigma)$ is defined to be an underlying digraph $D=(\mathcal{V}, \mathcal{A})$ and a signed function $\sigma: \mathcal{A} \rightarrow\{1, -1\}$. An arc from a vertex $u$ to a vertex $v$ is denoted to be $uv$. A positive (negative) arc of $S$ is defined to be an arc with $+1$ $ (-1)$ sign and the product of signs of its arcs is called to be the sign of a sidigraph. A sidigraph is called to be positive (negative) if the sign of the sidigraph is positive (negative).

A directed path $P_n$ is defined to be a vertex set $\{v_i|i=1, 2, \ldots, n\}$ and a signed arc set $\{v_iv_{i+1}|i=1, 2, \ldots, n-1\}$. A directed cycle $C_n$ is defined to be a vertex set $\{v_i|i=1, 2, \ldots, n\}$ and a signed arc set $\{v_iv_{i+1}|i=1, 2, \ldots, n-1\}\cup\{v_nv_1\}$. A sidigraph is called to be cycle-balanced (non cycle-balanced) if each directed cycles of the sidigraph is positive (negative). And the positive (negative) directed cycle is denoted to be $C_n^+$ ($C_n^-$), where $n$ is the number of vertices of the directed cycle. A sidigraph is connected if its underlying graph is connected. A unicyclic sidigraph is defined to be a connected sidigraph containing a unique singed directed cycle and a bicyclic sidigraph is defined to be a connected sidigraph containing exactly two singed directed cycles.

A adjacency matrix $A(S)=(a_{ij})_{n \times n}$ of a sidigraph $S$ with $\mathcal{V}=\{v_1, v_2, \ldots, v_n$\} is given by
$$a_{ij}=
\begin{cases}
\sigma(v_i, v_j),\ \ \mbox{if} \ v_iv_j\in\mathcal{A},\\
0,\ \ \ \ \ \ \ \ \ \ \ \mbox{otherwise}.
\end{cases}$$
Let $\phi_{S}(x)=det(xI_n-A(S))$ be the characteristic polynomial of $S$ and $z_1, z_2, \ldots, z_n$ be the eigenvalues of $A(S)$.

In \cite{PiBh}, the energy of sidigraph was defined as the sum of the absolute values of the real parts of its eigenvalues of the adjacency matrix of a sidigraph, then
$$E(S)=\sum_{k=1}^n|\textnormal{Re}(z_k)|,$$
where $\textnormal{Re}(z_k)$ is the real part of eigenvalue $z_k$. And the energy for a positive (negative) directed cycle $C_n^+$ ($C_n^-$) is
$$E(C_n^+)=\sum_{k=0}^{n-1}|\cos{\frac{2k\pi}{n}}|,\ E(C_n^-)=\sum_{k=0}^{n-1}|\cos{\frac{(2k+1)\pi}{n}}|.$$
Furthermore, some results of energy of sidigraphs can be found in \cite{PeRa}:
$$E(C_r^+)\leq E(C_s^+),\ E(C_r^-)\leq E(C_s^-).$$
where $C_r^+$ ($C_r^-$) and $C_s^+$ ($C_s^-$) are two positive (negative) directed cycles and $2\leq r\leq s$.

In \cite{PiBh}, we also have the formulae to calculate energy of a positive (negative) directed cycle $C_n^+$ ($C_n^-$), where $n\geq 2$.

\begin{equation}\label{eq:ch-1}
E(C_n^+)=
\begin{cases}
2\cot{\frac{\pi}{n}},\ \ \textnormal{if}\ n\equiv0\ (\textnormal{mod}\ 4),\\
2\csc{\frac{\pi}{n}},\ \ \textnormal{if}\ n\equiv2\ (\textnormal{mod}\ 4),\\
\csc{\frac{\pi}{2n}},\ \ \ \textnormal{if}\ n\equiv1\ (\textnormal{mod}\ 2).
\end{cases}
\end{equation}

\begin{equation}\label{eq:ch-2}
E(C_n^-)=
\begin{cases}
2\csc{\frac{\pi}{n}},\ \ \textnormal{if}\ n\equiv0\ (\textnormal{mod}\ 4),\\
2\cot{\frac{\pi}{n}},\ \ \textnormal{if}\ n\equiv2\ (\textnormal{mod}\ 4),\\
\csc{\frac{\pi}{2n}},\ \ \ \textnormal{if}\ n\equiv1\ (\textnormal{mod}\ 2).
\end{cases}
\end{equation}

Because the adjacency matrix $A(S)$ of a sigraph is not a symmetric matrix, so the eigenvalues of $S$ may be have the real part and the imaginary part. In 2019,  Farooq et al. \cite{FaKC} considered the iota energy of sidigraph which defined as the sum of the absolute values of the imaginary parts of its eigenvalues of the adjacency matrix of a sidigraph, then
$$E_c(S)=\sum_{k=1}^n|\textnormal{Im}(z_k)|,$$
where $\textnormal{Im}(z_k)$ is the imaginary part of eigenvalue $z_k$.
Similarly to energy, we have
$$E_c(C_n^+)=\sum_{k=0}^{n-1}|\sin{\frac{2k\pi}{n}}|,\ E_c(C_n^-)=\sum_{k=0}^{n-1}|\sin{\frac{(2k+1)\pi}{n}}|.$$
And in \cite{KhFR}, we have
$$E_c(C_r^+)\leq E_c(C_s^+),\ E_c(C_r^-)\leq E_c(C_s^-).$$
We also have the formulae in \cite{FaKC} to calculate iota energy of a positive (negative) directed cycle $C_n^+$ ($C_n^-$), where $n\geq 2$.

\begin{equation}\label{eq:ch-3}
E_c(C_n^+)=
\begin{cases}
2\cot{\frac{\pi}{n}},\ \ \textnormal{if}\ n\equiv0\ (\textnormal{mod}\ 2),\\
\cot{\frac{\pi}{2n}},\ \ \ \textnormal{if}\ n\equiv1\ (\textnormal{mod}\ 2).
\end{cases}
\end{equation}

\begin{equation}\label{eq:ch-4}
E_c(C_n^-)=
\begin{cases}
2\csc{\frac{\pi}{n}},\ \ \textnormal{if}\ n\equiv0\ (\textnormal{mod}\ 2),\\
\cot{\frac{\pi}{2n}},\ \ \ \textnormal{if}\ n\equiv1\ (\textnormal{mod}\ 2).
\end{cases}
\end{equation}

A sidigraph $S=(D, \sigma)$ is strongly connected if it has a directed path from $u$ to $v$ and a directed path from $v$ to $u$ ($\forall u,v\in \mathcal{V}$). A strong component of a sidigraph $S$ is the maximal strongly connected subsidigraph of $S$. And the iota energy of a sidigraph $S$ is equal to the sum of the iota energy of every strong components of a sidigraph $S$. Thoughout this paper, we consider a special graph class for bicyclic signed digraphs $\mathcal{S}_n$ with $n$ vertices ($n\geq4$) which have two vertex-disjoint signed directed even cycles. Let two vertex-disjoint signed directed even cycles be $C_{r_1}$ and $C_{r_2}$, where $r_1, r_2$ are all even numbers and $2\leq r_1, r_2\leq n-2$. Moreover, $C_{r_1}$ and $C_{r_2}$ are the strong components of $S\in \mathcal{S}_n$, obviously,
$$E_c(S)=E_c(C_{r_1})+E_c(C_{r_2}).$$

We also need some lemmas about linear combinations of trigonometric functions to prove our main results.

\noindent\begin{lemma}\label{le:ch-2.1}\cite{FaKA} The function $f(x)=2\cot{\frac{\pi}{x}}+2\cot{\frac{\pi}{n-x}}$ is increasing on $[2,\frac{n}{2}]$ and decreasing on $[\frac{n}{2},n-2]$, where $n>4$. \end{lemma}

\noindent\begin{lemma}\label{le:ch-2.2}\cite{YaWa} The function $f(x)=2\csc{\frac{\pi}{x}}+2\csc{\frac{\pi}{n-x}}$ is decreasing on $[2,\frac{n}{2}]$, where $n>4$. \end{lemma}

\noindent\begin{lemma}\label{le:ch-2.3}\cite{YaWa} The function $f(x)=2\csc{\frac{\pi}{x}}+2\cot{\frac{\pi}{n-x}}$ is decreasing on $[2,n-2]$, where $n>4$. \end{lemma}

\noindent\begin{lemma}\label{le:ch-2.4}  The function $f(x)=\frac{\pi}{x^2}{\csc}^2{\frac{\pi}{x}}$ is decreasing on $[2,n-2]$, where $n>4$. \end{lemma}
\begin{proof}
Let $g(t)=t^2{\csc}^2{t}$, where $t=\frac{\pi}{x^2}$, then $t \in [\frac{\pi}{n-2},\frac{\pi}{2}]$.
We have $g'(t)=2t{\csc}^2{t}-2t^2\cos{t}{\csc}^3{t}=2t{\csc}^3{t}(\sin{t}-t\cos{t})$. In order to prove $g(t)$ is increasing on $[\frac{\pi}{n-2},\frac{\pi}{2}]$, we should prove $\sin{t}-t\cos{t} \geq 0$.

Let $h(t)=\sin{t}-t\cos{t}$. In order to prove $h(t) \geq0$, we should prove $h(t)_\textnormal{min}:=\min\{h(t)|t \in [\frac{\pi}{n-2},\frac{\pi}{2}]\} \geq 0$. We have $h'(t)=\cos{t}-\cos{t}+t\sin{t}=t\sin{t}>0$. So $h(t)$ is increasing on $[\frac{\pi}{n-2},\frac{\pi}{2}]$. It implies $h(t) \geq h(t)_\textnormal{min}=h(\frac{\pi}{n-2})\rightarrow0$ $(n \rightarrow \infty)$. Hence $f(x)$ is decreasing on $[2,n-2]$.
\end{proof}

\section{Iota energy orderings of bicyclic signed digraphs}

In this section, we characterize two iota energy orderings of bicyclic sidigraphs, one is including two positive or two negative directed even cycles, the other is including one positive and one negative directed even cycles. In order to find the two orderings, we need consider two cases: $n$ is an even number or $n$ is an odd number. Firstly, we consider the iota energy of bicyclic sidigraphs with two positive or two negative directed even cycles.

\noindent\begin{lemma}\label{le:ch-3.1}
Let $C_m^+$ and $C_{n-m}^+$ ($C_m^-$ and $C_{n-m}^-$) be vertex-disjoint positive (negative) directed even cycles, where $n$ and $m$ are all even numbers, $n>4$ and $2\leq m \leq n-2$. Then we have

(i) If $\frac{n}{2}$ is an even number, then
\begin{equation}
\begin{split}\label{eq:ch-5}
&E_c(C_2^-)+E_c(C_{n-2}^-)>E_c(C_4^-)+E_c(C_{n-4}^-)>\cdots>E_c(C_{\frac{n}{2}}^-)+E_c(C_{\frac{n}{2}}^-)\\
&>E_c(C_{\frac{n}{2}}^+)+E_c(C_{\frac{n}{2}}^+)>\cdots>E_c(C_4^+)+E_c(C_{n-4}^+)>E_c(C_2^+)+E_c(C_{n-2}^+).
\end{split}
\end{equation}

(ii) If $\frac{n}{2}$ is an odd number, then
\begin{equation}
\begin{split}\label{eq:ch-6}
&E_c(C_2^-)+E_c(C_{n-2}^-)>E_c(C_4^-)+E_c(C_{n-4}^-)>\cdots>E_c(C_{\frac{n}{2}-1}^-)+E_c(C_{\frac{n}{2}+1}^-)\\
&>E_c(C_{\frac{n}{2}-1}^+)+E_c(C_{\frac{n}{2}+1}^+)>\cdots>E_c(C_4^+)+E_c(C_{n-4}^+)>E_c(C_2^+)+E_c(C_{n-2}^+).
\end{split}
\end{equation}
\end{lemma}
\begin{proof}
If $2\leq m \leq \frac{n}{2}$, then $\frac{n}{2} \leq n-m \leq n-2$, and $m, n-m$ are all even numbers. And vice versa, so we only consider the former.

Firstly, by Eq.\eqref{eq:ch-3} and Eq.\eqref{eq:ch-4}, we have
$$E_c(C_m^+)+E_c(C_{n-m}^+)=2\cot{\frac{\pi}{m}}+2\cot{\frac{\pi}{n-m}},$$
$$E_c(C_m^-)+E_c(C_{n-m}^-)=2\csc{\frac{\pi}{m}}+2\csc{\frac{\pi}{n-m}}.$$

Next, by Lemma \ref{le:ch-2.2}, we see that $2\csc{\frac{\pi}{m}}+2\csc{\frac{\pi}{n-m}}$ is decreasing on the interval $[2,\frac{n}{2}]$. And by Lemma \ref{le:ch-2.1}, we see that $2\cot{\frac{\pi}{m}}+2\cot{\frac{\pi}{n-m}}$ is increasing on the interval $[2,\frac{n}{2}]$.
So if $\frac{n}{2}$ is an even number, we have
$$E_c(C_2^-)+E_c(C_{n-2}^-)>E_c(C_4^-)+E_c(C_{n-4}^-)>\cdots>E_c(C_{\frac{n}{2}}^-)+E_c(C_{\frac{n}{2}}^-),$$
$$E_c(C_{\frac{n}{2}}^+)+E_c(C_{\frac{n}{2}}^+)>\cdots>E_c(C_4^+)+E_c(C_{n-4}^+)>E_c(C_2^+)+E_c(C_{n-2}^+).$$

Obviously,
$$2\csc{\frac{2\pi}{n}}+2\csc{\frac{2\pi}{n}}>2\cot{\frac{2\pi}{n}}+2\cot{\frac{2\pi}{n}},$$
so the ordering of \eqref{eq:ch-5} holds, the proof of \eqref{eq:ch-6} is similar.
\end{proof}

\noindent\begin{lemma}\label{le:ch-3.2}
If the bicyclic signed digraphs $\mathcal{S}_n$ with $n$ vertices ($n\geq22$) which have two vertex-disjoint signed directed even cycles, then we have

(i) If $n$ is an even number and $\frac{n}{2}$ is an odd number, then
$$E_c(C_{\frac{n-2}{2}}^-)+E_c(C_{\frac{n-2}{2}}^-)>E_c(C_2^+)+E_c(C_{n-2}^+)>E_c(C_{\frac{n-2}{2}}^+)+E_c(C_{\frac{n-2}{2}}^+).$$

(ii) If $n$ is an even number and $\frac{n}{2}$ is an even number, then
$$E_c(C_{\frac{n-2}{2}-1}^-)+E_c(C_{\frac{n-2}{2}+1}^-)>E_c(C_2^+)+E_c(C_{n-2}^+)>E_c(C_{\frac{n-2}{2}-1}^+)+E_c(C_{\frac{n-2}{2}+1}^+).$$

(iii) If $n$ is an even number, then
$$E_c(C_6^+)+E_c(C_{n-6}^+)>E_c(C_2^-)+E_c(C_{n-4}^-)>E_c(C_4^+)+E_c(C_{n-4}^+).$$
\end{lemma}
\begin{proof}
(i) By \eqref{eq:ch-3} and \eqref{eq:ch-4}, in order to prove $$E_c(C_{\frac{n-2}{2}}^-)+E_c(C_{\frac{n-2}{2}}^-)>E_c(C_2^+)+E_c(C_{n-2}^+)>E_c(C_{\frac{n-2}{2}}^+)+E_c(C_{\frac{n-2}{2}}^+),$$
we should prove
$$2\csc{\frac{2\pi}{n-2}}+2\csc{\frac{2\pi}{n-2}}>2\cot{\frac{\pi}{n-2}}>2\cot{\frac{2\pi}{n-2}}+2\cot{\frac{2\pi}{n-2}}.$$
Since $$\frac{1}{\sin{\frac{\pi}{n-2}}\cos{\frac{\pi}{n-2}}}>\frac{{\cos}^2{\frac{\pi}{n-2}}}{\sin{\frac{\pi}{n-2}}\cos{\frac{\pi}{n-2}}}>\frac{2{\cos}^2{\frac{\pi}{n-2}}-1}{\sin{\frac{\pi}{n-2}}\cos{\frac{\pi}{n-2}}}.$$
So (i) is proved.

(ii) By \eqref{eq:ch-3} and \eqref{eq:ch-4}, in order to prove $$E_c(C_{\frac{n-2}{2}-1}^-)+E_c(C_{\frac{n-2}{2}+1}^-)>E_c(C_2^+)+E_c(C_{n-2}^+)>E_c(C_{\frac{n-2}{2}-1}^+)+E_c(C_{\frac{n-2}{2}+1}^+),$$
we should prove
$$2\csc{\frac{2\pi}{n-4}}+2\csc{\frac{2\pi}{n}}>2\cot{\frac{\pi}{n-2}}>2\cot{\frac{2\pi}{n-4}}+2\cot{\frac{2\pi}{n}}.$$
By Lemma \ref{le:ch-2.2}, $f(x)=2\csc{\frac{\pi}{x}}+2\csc{\frac{\pi}{n-x}}$ is decreasing on $[2,\frac{n}{2}]$, where $n>4$. Then $$2\csc{\frac{2\pi}{n-4}}+2\csc{\frac{2\pi}{n}}>2\csc{\frac{2\pi}{n-2}}+2\csc{\frac{2\pi}{n-2}}.$$
By Lemma \ref{le:ch-2.1}, $f(x)=2\cot{\frac{\pi}{x}}+2\cot{\frac{\pi}{n-x}}$ is increasing on $[2,\frac{n}{2}]$, where $n>4$. Then $$2\cot{\frac{2\pi}{n-4}}+2\cot{\frac{2\pi}{n}}<2\cot{\frac{2\pi}{n-2}}+2\cot{\frac{2\pi}{n-2}}.$$
And by (i), (ii) is proved.

(iii) By \eqref{eq:ch-3} and \eqref{eq:ch-4}, in order to prove
$$E_c(C_6^+)+E_c(C_{n-6}^+)>E_c(C_2^-)+E_c(C_{n-4}^-)>E_c(C_4^+)+E_c(C_{n-4}^+),$$
we should prove
$$2\cot{\frac{\pi}{6}}+2\cot{\frac{\pi}{n-6}}>2\csc{\frac{\pi}{2}}+2\csc{\frac{\pi}{n-4}}>2\cot{\frac{\pi}{4}}+2\cot{\frac{\pi}{n-4}}.$$
Obviously, $$2\csc{\frac{\pi}{2}}+2\csc{\frac{\pi}{n-4}}>2\cot{\frac{\pi}{4}}+2\cot{\frac{\pi}{n-4}}.$$
In order to prove
$$2\cot{\frac{\pi}{6}}+2\cot{\frac{\pi}{n-6}}>2\csc{\frac{\pi}{2}}+2\csc{\frac{\pi}{n-4}},$$
we should prove
$$2\csc{\frac{\pi}{n-4}}-2\cot{\frac{\pi}{n-6}}<2\sqrt{3}-2.$$

Let $f(x)=2\csc{\frac{\pi}{n-4}}-2\cot{\frac{\pi}{n-6}}$, then
$$f'(x)=2(\frac{\pi}{(n-4)^2}\cos{\frac{\pi}{n-4}}{\csc}^2{\frac{\pi}{n-4}}-\frac{\pi}{(n-6)^2}{\csc}^2{\frac{\pi}{n-6}}).$$
By Lemma \ref{le:ch-2.4}, we have $\frac{\pi}{x^2}{\csc}^2{\frac{\pi}{x}}$ is decreasing on $[2,n-2]$, where $n>4$. Since $n-4>n-6$,
then $$\frac{\pi}{(n-4)^2}\cos{\frac{\pi}{n-4}}{\csc}^2{\frac{\pi}{n-4}}<\frac{\pi}{(n-4)^2}{\csc}^2{\frac{\pi}{n-4}}<\frac{\pi}{(n-6)^2}{\csc}^2{\frac{\pi}{n-6}},$$
so $f'(x)<0$, $f(x)$ is decreasing on $[2,n-2]$, where $n\geq22$. It implies $f(x) \leq f(x)_\textnormal{max}=f(22)\approx 1.463 < 2\sqrt{3}-2$.
So (iii) is proved.
\end{proof}

Now, we give our main result that the ordering of bicyclic sidigraphs with iota energy including two positive or two negative directed even cycles.

\noindent\begin{theorem}\label{th:ch-3.3}
If the bicyclic signed digraphs $\mathcal{S}_n$ with $n$ vertices ($n\geq4$) which have two vertex-disjoint signed directed even cycles, then we give the whole ordering of bicyclic sidigraphs in $\mathcal{S}_n$ with iota energy including two positive or two negative directed even cycles.

(i) When $n$ is an even number, $\frac{n}{2}$ is an even number and $n\geq 22$, we have
\begin{equation*}
\begin{split}
&E_c(C_2^-)+E_c(C_{n-2}^-)>E_c(C_4^-)+E_c(C_{n-4}^-)>\cdots>E_c(C_{\frac{n}{2}}^-)+E_c(C_{\frac{n}{2}}^-)>E_c(C_{\frac{n}{2}}^+)\\
&+E_c(C_{\frac{n}{2}}^+)>\cdots>E_c(C_6^+)+E_c(C_{n-6}^+)>E_c(C_2^-)+E_c(C_{n-4}^-)>E_c(C_4^+)+E_c(C_{n-4}^+)\\
&>E_c(C_4^-)+E_c(C_{n-6}^-)>\cdots>E_c(C_{\frac{n-2}{2}-1}^-)+E_c(C_{\frac{n-2}{2}+1}^-)>E_c(C_2^+)+E_c(C_{n-2}^+)\\
&>E_c(C_{\frac{n-2}{2}-1}^+)+E_c(C_{\frac{n-2}{2}+1}^+)>\cdots>E_c(C_6^+)+E_c(C_{n-8}^+)>E_c(C_2^-)+E_c(C_{n-6}^-)\\
&>E_c(C_4^+)+E_c(C_{n-6}^+)>E_c(C_4^-)+E_c(C_{n-8}^-)>\cdots>E_c(C_{\frac{n-4}{2}}^-)+E_c(C_{\frac{n-4}{2}}^-)>E_c(C_2^+)\\
&+E_c(C_{n-4}^+)>\cdots>E_c(C_6^+)+E_c(C_{16}^+)>E_c(C_2^-)+E_c(C_{18}^-)>E_c(C_4^+)+E_c(C_{18}^+)\\
&>E_c(C_4^-)+E_c(C_{16}^-)>\cdots>E_c(C_{10}^-)+E_c(C_{10}^-)>E_c(C_2^+)+E_c(C_{20}^+).
\end{split}
\end{equation*}

(ii) When $n$ is an odd number, $\frac{n-1}{2}$ is an even number, and $n\geq 22$, we have
\begin{equation*}
\begin{split}
&E_c(C_2^-)+E_c(C_{n-3}^-)>E_c(C_4^-)+E_c(C_{n-5}^-)>\cdots>E_c(C_{\frac{n-1}{2}}^-)+E_c(C_{\frac{n-1}{2}}^-)>E_c(C_{\frac{n-1}{2}}^+)\\
&+E_c(C_{\frac{n-1}{2}}^+)>\cdots>E_c(C_6^+)+E_c(C_{n-7}^+)>E_c(C_2^-)+E_c(C_{n-5}^-)>E_c(C_4^+)+E_c(C_{n-5}^+)\\
&>E_c(C_4^-)+E_c(C_{n-7}^-)>\cdots>E_c(C_{\frac{n-3}{2}-1}^-)+E_c(C_{\frac{n-3}{2}+1}^-)>E_c(C_2^+)+E_c(C_{n-3}^+)\\
&>E_c(C_{\frac{n-3}{2}-1}^+)+E_c(C_{\frac{n-3}{2}+1}^+)>\cdots>E_c(C_6^+)+E_c(C_{n-9}^+)>E_c(C_2^-)+E_c(C_{n-7}^-)\\
&>E_c(C_4^+)+E_c(C_{n-7}^+)>E_c(C_4^-)+E_c(C_{n-9}^-)>\cdots>E_c(C_{\frac{n-5}{2}}^-)+E_c(C_{\frac{n-5}{2}}^-)>E_c(C_2^+)\\
&+E_c(C_{n-5}^+)>\cdots>E_c(C_6^+)+E_c(C_{16}^+)>E_c(C_2^-)+E_c(C_{18}^-)>E_c(C_4^+)+E_c(C_{18}^+)\\
&>E_c(C_4^-)+E_c(C_{16}^-)>\cdots>E_c(C_{10}^-)+E_c(C_{10}^-)>E_c(C_2^+)+E_c(C_{20}^+).
\end{split}
\end{equation*}

(iii) When $n$ is an even number, $\frac{n}{2}$ is an odd number and $n\geq 22$, we have
\begin{equation*}
\begin{split}
&E_c(C_2^-)+E_c(C_{n-2}^-)>E_c(C_4^-)+E_c(C_{n-4}^-)>\cdots>E_c(C_{\frac{n}{2}-1}^-)+E_c(C_{\frac{n}{2}+1}^-)\\
&>E_c(C_{\frac{n}{2}-1}^+)+E_c(C_{\frac{n}{2}+1}^+)>\cdots>E_c(C_6^+)+E_c(C_{n-6}^+)>E_c(C_2^-)+E_c(C_{n-4}^-)\\
&>E_c(C_4^+)+E_c(C_{n-4}^+)>E_c(C_4^-)+E_c(C_{n-6}^-)>\cdots>E_c(C_{\frac{n-2}{2}}^-)+E_c(C_{\frac{n-2}{2}}^-)\\
&>E_c(C_2^+)+E_c(C_{n-2}^+)>E_c(C_{\frac{n-2}{2}}^+)+E_c(C_{\frac{n-2}{2}}^+)>\cdots>E_c(C_6^+)+E_c(C_{n-8}^+)\\
&>E_c(C_2^-)+E_c(C_{n-6}^-)>E_c(C_4^+)+E_c(C_{n-6}^+)>E_c(C_4^-)+E_c(C_{n-8}^-)>\cdots\\
&>E_c(C_{\frac{n-4}{2}-1}^-)+E_c(C_{\frac{n-4}{2}+1}^-)>E_c(C_2^+)+E_c(C_{n-4}^+)>\cdots>E_c(C_6^+)+E_c(C_{16}^+)\\
&>E_c(C_2^-)+E_c(C_{18}^-)>E_c(C_4^+)+E_c(C_{18}^+)>E_c(C_4^-)+E_c(C_{16}^-)>\cdots>E_c(C_{10}^-)\\
&+E_c(C_{10}^-)>E_c(C_2^+)+E_c(C_{20}^+).
\end{split}
\end{equation*}

(iv) When $n$ is an odd number, $\frac{n-1}{2}$ is an odd number and $n\geq 22$, we have
\begin{equation*}
\begin{split}
&E_c(C_2^-)+E_c(C_{n-3}^-)>E_c(C_4^-)+E_c(C_{n-5}^-)>\cdots>E_c(C_{\frac{n-1}{2}-1}^-)+E_c(C_{\frac{n-1}{2}+1}^-)\\
&>E_c(C_{\frac{n-1}{2}-1}^+)+E_c(C_{\frac{n-1}{2}+1}^+)>\cdots>E_c(C_6^+)+E_c(C_{n-7}^+)>E_c(C_2^-)+E_c(C_{n-5}^-)\\
&>E_c(C_4^+)+E_c(C_{n-5}^+)>E_c(C_4^-)+E_c(C_{n-7}^-)>\cdots>E_c(C_{\frac{n-3}{2}}^-)+E_c(C_{\frac{n-3}{2}}^-)\\
&>E_c(C_2^+)+E_c(C_{n-3}^+)>E_c(C_{\frac{n-3}{2}}^+)+E_c(C_{\frac{n-3}{2}}^+)>\cdots>E_c(C_6^+)+E_c(C_{n-9}^+)\\
&>E_c(C_2^-)+E_c(C_{n-7}^-)>E_c(C_4^+)+E_c(C_{n-7}^+)>E_c(C_4^-)+E_c(C_{n-9}^-)>\cdots\\
&>E_c(C_{\frac{n-5}{2}-1}^-)+E_c(C_{\frac{n-5}{2}+1}^-)>E_c(C_2^+)+E_c(C_{n-5}^+)>\cdots>E_c(C_6^+)+E_c(C_{16}^+)\\
&>E_c(C_2^-)+E_c(C_{18}^-)>E_c(C_4^+)+E_c(C_{18}^+)>E_c(C_4^-)+E_c(C_{16}^-)>\cdots>E_c(C_{10}^-)\\
&+E_c(C_{10}^-)>E_c(C_2^+)+E_c(C_{20}^+).
\end{split}
\end{equation*}

(v) When $n\leq 22$, we have
\begin{equation*}
\begin{split}
&E_c(C_2^+)+E_c(C_{20}^+)>E_c(C_{10}^+)+E_c(C_{10}^+)>E_c(C_8^+)+E_c(C_{12}^+)>E_c(C_2^-)+E_c(C_{16}^-)\\
&>E_c(C_6^+)+E_c(C_{14}^+)>E_c(C_4^+)+E_c(C_{16}^+)>E_c(C_4^-)+E_c(C_{14}^-)>E_c(C_6^-)+E_c(C_{12}^-)\\
&>E_c(C_8^-)+E_c(C_{10}^-)>E_c(C_2^+)+E_c(C_{18}^+)>E_c(C_2^-)+E_c(C_{14}^-)>E_c(C_8^+)+E_c(C_{10}^+)\\
&>E_c(C_6^+)+E_c(C_{12}^+)>E_c(C_4^+)+E_c(C_{14}^+)>E_c(C_4^-)+E_c(C_{12}^-)>E_c(C_6^-)+E_c(C_{10}^-)\\
&>E_c(C_8^-)+E_c(C_8^-)>E_c(C_2^+)+E_c(C_{16}^+)>E_c(C_2^-)+E_c(C_{12}^-)>E_c(C_8^+)+E_c(C_8^+)\\
&>E_c(C_6^+)+E_c(C_{10}^+)>E_c(C_4^+)+E_c(C_{12}^+)>E_c(C_4^-)+E_c(C_{10}^-)>E_c(C_6^-)+E_c(C_8^-)\\
&>E_c(C_2^+)+E_c(C_{14}^+)>E_c(C_2^-)+E_c(C_{10}^-)>E_c(C_6^+)+E_c(C_8^+)>E_c(C_4^+)+E_c(C_{10}^+)\\
&>E_c(C_4^-)+E_c(C_8^-)>E_c(C_6^-)+E_c(C_6^-)>E_c(C_2^+)+E_c(C_{12}^+)>E_c(C_2^-)+E_c(C_8^-)\\
&>E_c(C_6^+)+E_c(C_6^+)>E_c(C_4^+)+E_c(C_8^+)=E_c(C_4^-)+E_c(C_6^-)>E_c(C_2^+)+E_c(C_{10}^+)\\
&>E_c(C_2^-)+E_c(C_6^-)>E_c(C_4^-)+E_c(C_4^-)>E_c(C_4^+)+E_c(C_6^+)>E_c(C_2^+)+E_c(C_8^+)\\
&=E_c(C_2^-)+E_c(C_4^-)>E_c(C_4^+)+E_c(C_4^+)=E_c(C_2^-)+E_c(C_2^-)>E_c(C_2^+)+E_c(C_6^+)\\
&>E_c(C_2^+)+E_c(C_4^+)>E_c(C_2^+)+E_c(C_2^+).
\end{split}
\end{equation*}
\end{theorem}
\begin{proof}
By Lemma \ref{le:ch-3.1} and Lemma \ref{le:ch-3.2}, we get the orderings (i) and (iii) of bicyclic sidigraphs in $\mathcal{S}_n$ with iota energy. The proof of ordering (ii) ((iv)) are similar to (i) ((iii)) by changing $n$ in (i) ((iii)) into $n-1$.
\end{proof}

Next, we consider the iota energy of bicyclic sidigraphs with one positive and one negative directed even cycles.

\noindent\begin{lemma}\label{le:ch-3.4}
Let $C_m^+$ and $C_{n-m}^+$ ($C_m^-$ and $C_{n-m}^-$) be vertex-disjoint positive (negative) directed even cycles, where $n$ and $m$ are all even numbers, $n> 4$ and $2 \leq m \leq n-2$. Then we have
\begin{equation}
\begin{split}\label{eq:ch-7}
E_c(C_2^-)+E_c(C_{n-2}^+)>E_c(C_4^-)+E_c(C_{n-4}^+)>\cdots>E_c(C_{n-4}^-)+E_c(C_4^+)>E_c(C_{n-2}^-)+E_c(C_2^+).
\end{split}
\end{equation}
\end{lemma}
\begin{proof}
Firstly, by Eq.\eqref{eq:ch-3} and Eq.\eqref{eq:ch-4}, we have
$$E_c(C_m^+)+E_c(C_{n-m}^-)=2\cot{\frac{\pi}{m}}+2\csc{\frac{\pi}{n-m}},$$
$$E_c(C_m^-)+E_c(C_{n-m}^+)=2\csc{\frac{\pi}{m}}+2\cot{\frac{\pi}{n-m}}.$$

Next, by Lemma \ref{le:ch-2.3}, $f(x)=2\csc{\frac{\pi}{x}}+2\cot{\frac{\pi}{n-x}}$ is decreasing on $[2,n-2]$, where $n>4$. Similarly, $f(x)=2\cot{\frac{\pi}{x}}+2\csc{\frac{\pi}{n-x}}$ is increasing on $[2,n-2]$, where $n>4$. So we only consider $E_c(C_m^-)+E_c(C_{n-m}^+)=2\csc{\frac{\pi}{m}}+2\cot{\frac{\pi}{n-m}}$, where $n>4$ and $2 \leq m \leq n-2$. Hence the ordering of \eqref{eq:ch-7} holds.
\end{proof}

\noindent\begin{remark}\label{re:ch-3.5} By calculation, for $n$ is an even number, we find

(i) If $10 \leq n \leq 16$, then
$$E_c(C_2^-)+E_c(C_{n-4}^+)>E_c(C_{n-2}^-)+E_c(C_2^+)>E_c(C_4^-)+E_c(C_{n-6}^+);$$

(ii) If $18 \leq n \leq 22$, then
$$E_c(C_4^-)+E_c(C_{n-6}^+)>E_c(C_{n-2}^-)+E_c(C_2^+)>E_c(C_6^-)+E_c(C_{n-8}^+);$$

(iii) If $24 \leq n \leq 30$, then
$$E_c(C_6^-)+E_c(C_{n-8}^+)>E_c(C_{n-2}^-)+E_c(C_2^+)>E_c(C_8^-)+E_c(C_{n-10}^+);$$

(iv) If $32 \leq n \leq 38$, then
$$E_c(C_8^-)+E_c(C_{n-10}^+)>E_c(C_{n-2}^-)+E_c(C_2^+)>E_c(C_{10}^-)+E_c(C_{n-12}^+);$$

(v) If $40 \leq n \leq 48$, then
$$E_c(C_{10}^-)+E_c(C_{n-12}^+)>E_c(C_{n-2}^-)+E_c(C_2^+)>E_c(C_{12}^-)+E_c(C_{n-14}^+);$$

And so on.

We can't find an uniform rule, so the ordering of bicyclic sidigraphs in $\mathcal{S}_n$ with iota energy including one positive and one negative directed even cycles, we ignore $E_c(C_2^+)+E_c(C_{n-2}^-)$.
\end{remark}

Now, we give our main result that the ordering of bicyclic sidigraphs with iota energy including one positive and one negative directed even cycles.

\noindent\begin{theorem}\label{th:ch-3.6}
If the bicyclic signed digraphs $\mathcal{S}_n$ with $n$ vertices ($n\geq4$) which have two vertex-disjoint signed directed even cycles, then except $C_2^+$, we give the ordering of bicyclic sidigraphs in $\mathcal{S}_n$ with iota energy including one positive and one negative directed even cycles.

(i) When $n$ is an even number, we have
\begin{equation*}
\begin{split}
&E_c(C_2^-)+E_c(C_{n-2}^+)>E_c(C_4^-)+E_c(C_{n-4}^+)>\cdots>E_c(C_{n-4}^-)+E_c(C_4^+)\\
&>E_c(C_2^-)+E_c(C_{n-4}^+)>E_c(C_4^-)+E_c(C_{n-6}^+)>\cdots>E_c(C_{n-6}^-)+E_c(C_4^+)\\
&>E_c(C_2^-)+E_c(C_{n-6}^+)>\cdots>E_c(C_2^-)+E_c(C_2^+).
\end{split}
\end{equation*}

(ii) When $n$ is an odd number, we have
\begin{equation*}
\begin{split}
&E_c(C_2^-)+E_c(C_{n-3}^+)>E_c(C_4^-)+E_c(C_{n-5}^+)>\cdots>E_c(C_{n-5}^-)+E_c(C_4^+)\\
&>E_c(C_2^-)+E_c(C_{n-5}^+)>E_c(C_4^-)+E_c(C_{n-7}^+)>\cdots>E_c(C_{n-7}^-)+E_c(C_4^+)\\
&>E_c(C_2^-)+E_c(C_{n-7}^+)>\cdots>E_c(C_2^-)+E_c(C_2^+).
\end{split}
\end{equation*}
\end{theorem}
\begin{proof}
Obviously, we have $2\csc{\frac{\pi}{n-4}}+2\cot{\frac{\pi}{4}}>2\csc{\frac{\pi}{2}}+2\cot{\frac{\pi}{n-4}}$, together with Lemma \ref{le:ch-3.4} and Remark \ref{re:ch-3.5}, we get the ordering (i) of bicyclic sidigraphs in $\mathcal{S}_n$ with iota energy. The proof of ordering (ii) is similar to (i) by changing $n$ in (i) into $n-1$.
\end{proof}

From Theorem \ref{th:ch-3.3} and Theorem \ref{th:ch-3.6}, We can get next Theorem.

\noindent\begin{theorem}\label{th:ch-3.7}
If the bicyclic signed digraphs $\mathcal{S}_n$ with $n$ vertices ($n\geq4$) which have two vertex-disjoint signed directed even cycles $C_{r_1}$ and $C_{r_2}$, where $2\leq r_i\leq {n-2}$ and $r_i$ is an even number $(i=1,2)$, then we have

(i) When $n$ is an even number, $S$ has maximal iota energy if the two vertex-disjoint signed directed even cycles are $C_2^-$ and $C_{n-2}^-$.

(ii) When $n$ is an odd number, $S$ has maximal iota energy if the two vertex-disjoint signed directed even cycles are $C_2^-$ and $C_{n-3}^-$.

(iii) $S$ has minimal iota energy if the two vertex-disjoint signed directed even cycles are both $C_2^+$.
\end{theorem}
\begin{proof}
(i) When $n$ is an even number, by (i) and (iii) of Theorem \ref{th:ch-3.3}, the maximal iota energy is $E_c(C_2^-)+E_c(C_{n-2}^-)=2\csc{\frac{\pi}{2}}+2\csc{\frac{\pi}{n-2}}$, and by (i) of Theorem \ref{th:ch-3.6}, the maximal iota energy is $E_c(C_2^-)+E_c(C_{n-2}^+)=2\csc{\frac{\pi}{2}}+2\cot{\frac{\pi}{n-2}}$. Obviously, $2\csc{\frac{\pi}{2}}+2\csc{\frac{\pi}{n-2}}>2\csc{\frac{\pi}{2}}+2\cot{\frac{\pi}{n-2}}$, so $S$ has maximal iota energy if the two vertex-disjoint signed directed even cycles are $C_2^-$ and $C_{n-2}^-$.

The proofs of (ii) and (iii) are analogous.
\end{proof}

\noindent\begin{remark}\label{re:ch-3.8}
Actually, by Theorems 2.12 of \cite{FaCK}, we can get more general situation of (i) and (iii) of Theorem \ref{th:ch-3.7}.
\end{remark}

Finally, we give an example about two orderings and the extremal iota energy of bicyclic signed digraphs $\mathcal{S}_{27}$.

\noindent\begin{example}\label{ex:ch-3.9}
If the bicyclic signed digraphs $\mathcal{S}_{27}$ with $27$ vertices which have two vertex-disjoint signed directed even cycles $C_{r_1}$ and $C_{r_2}$, where $2\leq r_i\leq {25}$ and $r_i$ is an even number $(i=1,2)$. Then the whole ordering of bicyclic sidigraphs in $\mathcal{S}_{27}$ with iota energy including two positive or two negative directed even cycles see Figure \ref{fi:ch-1}. And the ordering of bicyclic sidigraphs in $\mathcal{S}_{27}$ with iota energy including one positive and one negative directed even cycles (except $C_2^+$) see Figure \ref{fi:ch-2}.

Moreover, when the two vertex-disjoint signed directed even cycles are $C_2^-$ and $C_{24}^-$, the maximal iota energy of $\mathcal{S}_{27}$ is $E_c(C_2^-)+E_c(C_{24}^-)=17.32$. When the two vertex-disjoint signed directed even cycles are both $C_2^+$, the minimal iota energy of $\mathcal{S}_{27}$ is $E_c(C_2^+)+E_c(C_2^+)=0$.
\end{example}

\begin{figure}[htbp]
\begin{centering}
\includegraphics[scale=0.45]{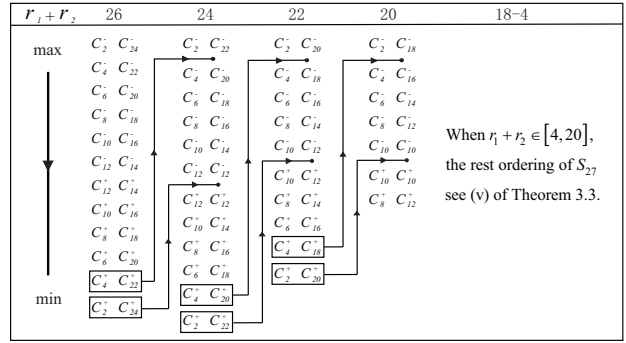}
\caption{The whole ordering of bicyclic sidigraphs in $\mathcal{S}_{27}$ with iota energy including two positive or two negative directed even cycles}\label{fi:ch-1}
\end{centering}
\end{figure}
\begin{figure}[htbp]
\begin{centering}
\includegraphics[scale=0.45]{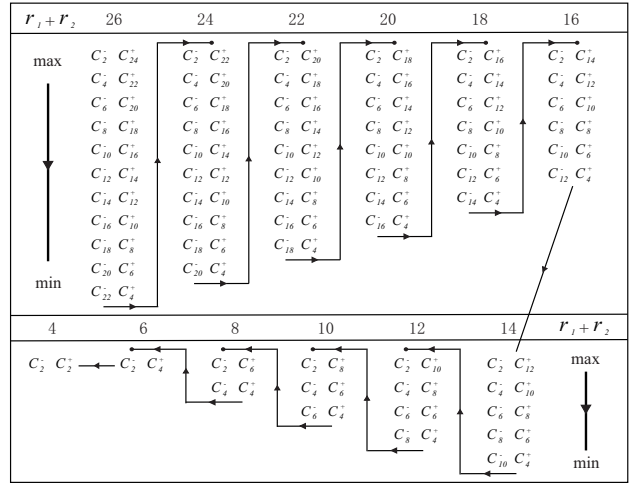}
\caption{The ordering of bicyclic sidigraphs in $\mathcal{S}_{27}$ with iota energy including one positive and one negative directed even cycles (except $C_2^+$)}\label{fi:ch-2}
\end{centering}
\end{figure}

\end{document}